\documentclass[11pt,a4paper]{amsart}
\usepackage{fullpage}
\usepackage{amssymb,amscd,hyperref}

\usepackage[all]{xy}
\usepackage{textcomp}
\usepackage{tikz}
\usepackage{color}
\usetikzlibrary{shapes,arrows}
\usepackage{cleveref}

\newtheorem{thm}{Theorem}[section]
\newtheorem{lem}[thm]{Lemma}
\newtheorem{prop}[thm]{Proposition}
\newtheorem{cor}[thm]{Corollary}

\theoremstyle{definition}
\newtheorem{defn}[thm]{Definition}

\theoremstyle{remark}
\newtheorem{rem}[thm]{Remark}
\newtheorem{notation}[thm]{Notation}

\newcommand{\caselist}[2]{\noindent\textsc{Case #1:} #2}
\newcommand{\mathand}{\qquad\hbox{and}\qquad}
\newcommand{\F}{\mathbb{F}}

\def\[#1\]{\begin{align*}#1\end{align*}}
\def\fp{\mathbb{F}_p}
\def\pl{{p^\ell}}
\def\fpxl{\fp[x]/x^\pl}
\def\Tor{\mathrm{Tor}}

\def\xra{\xrightarrow}

\def\Z{\mathbb{Z}}

\def\F{\mathbb{F}}

\def\leq{\leqslant}
\def\geq{\geqslant}

\def\inc{\mathrm {inc}}

\def\ra{\rightarrow}

\def\B{{\sf B}}

\def\HH{{\sf HH}}
\def\THH{{\sf THH}}
\def\TAQ{{\sf TAQ}}

\def\id{\mathrm{id}}
\def\sgn{\mathrm{sgn}}
\newcommand{\G}{\Gamma}

\begin{document}
\title{On the higher topological Hochschild homology of $\F_p$ and
 commutative $\F_p$-group algebras}
\author{Irina Bobkova}
\address{Mathematics Department, Northwestern University, 2033
  Sheridan Road, Evanston, IL 60208-2730, USA}
\email{bobkova@math.northwestern.edu }

\author{Ayelet Lindenstrauss}
\address{Mathematics Department, Indiana University, Bloomington, IN 47405, USA}
\email{alindens@indiana.edu}

\author{Kate Poirier}
\address{Mathematics Department, New York City
College of Technology, CUNY, 300 Jay Street, Brooklyn, NY 11201, USA}
\email{kpoirier@citytech.cuny.edu}

\author{Birgit Richter}
\address{Fachbereich Mathematik der Universit\"at Hamburg,
Bundesstra{\ss}e 55, 20146 Hamburg, Germany}
\email{birgit.richter@uni-hamburg.de}

\author{Inna Zakharevich}
\address{School of Mathematics, Institute for Advanced Study, 
  Einstein Drive, Princeton, NJ 08540, USA}
\email{zakharevich@ias.edu}

\date{\today}
\begin{abstract}
We extend Torleif Veen's calculation of higher topological Hochschild
homology $\THH^{[n]}_*(\F_p) $ from $n\leq 2p$ to $n\leq 2p+2$ for $p$
odd, and from $n=2$   to $n\leq 3$ for $p=2$.
We calculate higher Hochschild homology $\HH_*^{[n]}(k[x])$ over $k$
for any integral domain $k$, and $\HH_*^{[n]}(\fp[x]/x^{p^\ell})$ for
all $n>0$.
We use this and \'etale
descent to calculate $\HH_*^{[n]}(\fp[G])$ for all $n>0$ for any
cyclic group $G$, and therefore also for any finitely generated 
abelian group $G$.
We show a splitting result for higher $\THH$ of
commutative $\F_p$-group algebras and use this technique to calculate higher
topological Hochschild homology of such group
algebras for as large an $n$ as $\THH^{[n]}_*(\F_p) $ is known for.
\end{abstract}
\thanks{The last named author was partially supported by the Institute
  for Advanced Study}
\maketitle

\section{Introduction}
Given a commutative ring $R$ and an $R$-module $M$, Jean-Louis Loday
introduced a functor $\mathcal{L}(R,M)$ which takes a based simplicial
set $X.$ to the simplicial $R$-module which consists in degree $n$ of
$M$ tensored with one copy of $R$ for each element in $X_n\setminus\{
* \}$.  The homotopy groups of the image of the Loday functor turn out
to be independent of the simplicial structure used for $X.$; they depend only
on its homotopy type.

Applying
this functor to the usual simplicial model of $\mathbb{S}^1$ with one
non-degenerate
$0$-cell and one non-degenerate $1$-cell, we get the classical
Hochschild complex whose homology is $\HH_*(R;M)$.  Extending this,
the higher topological Hochschild homology groups $\HH^{[n]} _*(R;M)$
were defined by Teimuraz Pirashvili \cite{P} as the homotopy groups of
$\mathcal{L}(R,M)$ evaluated on $\mathbb{S}^n$.

Morten Brun, Gunnar Carlsson, and Bj{\o}rn Dundas introduced a
topological version
of $\mathcal{L}(R,M)$ for a ring spectrum $R$ and an $R$-module
spectrum $M$ \cite{bcd}.
When evaluated on $\mathbb{S}^n$, it yields the spectrum $\THH^{[n]} (R;M)$,
the higher topological Hochschild homology of $R$ with coefficients in
$M$. For $M=R$ with the obvious action by
multiplication  $M$ is omitted from the notation.

\smallskip
Higher (topological) Hochschild homology features in several different
contexts.
There are stabilization maps in the algebraic context
$$ \HH^{[1]}_*(R) \ra  \HH^{[2]}_{*+1}(R) \ra \ldots \ra H\Gamma_{*-1}(R)$$
starting with Hochschild homology and ending with Gamma homology in the sense
of Alan Robinson and Sarah Whitehouse \cite{RW}. In the topological setting
they start
with $\THH(R)$ and end in topological Andr\'e-Quillen homology, $\TAQ(R)$,
$$\THH^{[1]}_*(R) \ra \THH^{[2]}_{*+1}(R) \ra \ldots \ra \TAQ_{*-1}(R).$$ 
The $k$-invariants of commutative ring spectra live in topological
Andr\'e-Quillen cohomology \cite{Ba} and obstructions for
$E_\infty$-ring structures on spectra live in Gamma cohomology
\cite{Ro,GH}, so these two cohomology theories are of great interest.

The evaluation of the Loday functor on higher dimensional tori is the same as
iterated topological Hochschild homology and this features in the
program for detecting red-shift in algebraic K-theory. Calculations of iterated
topological Hochschild homology use higher $\THH$ as an important ingredient.

Work of Benoit Fresse \cite{F} identifies Hochschild homology of order $n$
(in the
disguise of $E_n$-homology) with the homology groups of an algebraic $n$-fold
bar construction, thus $\HH^{[n]}_{*}(R)$ can be viewed as the homology of an
$n$-fold algebraic delooping. 

\smallskip

In his thesis Torleif Veen \cite{Vthesis,V} used a decomposition result for
$\mathcal{L}(R,M)$ to calculate $\THH^{[n]}_*(\F_p)=\pi_*( \THH^{[n]} (\F_p))$
for all $n\leq 2p$ and any odd prime $p$. 
For small $n$ such calculations were earlier done by John Rognes. 
Veen inductively sets up a
spectral sequence of Hopf algebras calculating
 $\THH^{[n]} _*(\F_p)$ from  $\THH^{[n-1]} _*(\F_p)$ with the base
 case  $\THH^{[1]} _*(\F_p)$ being known by work of Marcel B\"okstedt
 \cite{B}. 
 Veen explains why the spectral sequence has to collapse for $n\leq
 2p$.  By a careful 
 analysis of the structure of the spectral sequence, motivated by
 computer calculations, we show that it actually collapses 
for $n\leq 2p+2$ (Proposition \ref{prop:pushing}),
 thus getting a calculation of
 $\THH^{[n]} _*(\F_p)$ for those $n$.  The computer analysis  also found potential nontrivial differentials in the spectral sequence
 when $n=2p+3$.
 We actually believe that the differential will end up vanishing for
 all $n$.  We intend to return to this question in a future paper with
 Maria Basterra and Michael Mandell.   At $p=2$ Veen calculates 
 $\THH^{[n]}_*(\F_2)$ up to
 $n=2$. We include the $n=3$ case and also show that the generator in
$\THH_2(\F_2)$ stabilizes to a non-trivial element in the first
topological Andr\'e-Quillen homology group of $\F_2$ (Proposition
\ref{prop:stable}).

 We prove that for an $\F_p$-algebra $A$ and an  abelian group $G$,
$$\THH^{[n]}_*(A[G]) \cong \THH^{[n]}_* (A)\otimes \HH^{[n]}_* (\F_p[G]).$$
Using this, we calculate $\THH^{[n]}_*(\F_p[G])$ for any finitely
generated abelian group $G$ for $n\leq 2p+2$.  To extend this to
general abelian groups, observe that higher Hochschild homology
commutes with direct
limits. 

The actual calculations of higher Hochschild homology that we do are
of  $\HH^{[n]}_*(\fp[x])$ and of $\HH^{[n]}_*(\fp[x]/x^m)$ for any
  $m$. 

\medskip
We thank the Clay Mathematical Institute and the Banff
International Research Station for their support and hospitality.  We
would like to thank Michael Mandell for a very useful
conversation, and the referee for her or his careful reading which
caught an embarrassing blunder in an earlier draft of the paper. Our
warm thanks go the organizers of the BIRS workshop 
\emph{WIT: Women in Topology 2013}, Maria Basterra, Kristine Bauer,
Kathryn Hess and Brenda Johnson.
\section{Comparing the bar construction and its homology for some
  basic algebras }\label{sec:lemmas} 

We consider the two-sided bar construction $\B(k,A,k)$ where $k$ is a
commutative ring and $A=k[x]$ or $A=k[x]/x^m$.  The generator $x$ will be
allowed to be of any even degree; if $A=k[x]/x^2$ or $2=0$ in $k$, $x$ can be
of any degree.  Note that since $k$ is commutative, $A$ is also a graded
commutative ring, and so $\B(k,A,k)$ is a differential graded
augmented commutative $k$-algebra, with multiplication given by the
shuffle product. 

 Our goal in this section is to establish quasi-isomorphisms between
$\B(k,A,k)$ and its homology ring $\Tor^A_*(k,k)$ which are maps of
differential graded augmented $k$-algebras.  (We use the zero differential
on the homology ring.)  The quasi-isomorphisms are adapted from \cite{LL},
where similar maps are studied on the Hochschild complex for variables $x$
which have to be of degree
zero, but may satisfy other monic polynomial equations.  The reason that we
need these quasi-isomorphisms is that in Section \ref{sec:hhh} we will be
looking at iterated bar constructions of the form $\B(k, \B(k,A,k), k)$.  If
we know that there is some differential graded algebra $C$ with
quasi-isomorphisms that are algebra maps between $\B(k,A,k)$ and $C$, we
then get quasi-isomorphisms that are algebra maps between $\B(k, \B(k,A,k),
k)$ and $\B(k, C, k)$.
In the cases we study, the rings $C=\Tor^A_*(k,k)$ are very simple, and in
fact involve rings of the form of the $A$'s we deal with in this section, or
tensor products of them. Thus the $\B(k, C, k)$ can again be compared to
simpler graded algebras, and the process can continue.

The following propositions also re-prove what $\Tor^A_*(k,k)$  is for the
$A$'s we are interested in, but those are old and familiar results; our
motivation is understanding the bar complex $\B(k,A,k)$ as a differential
graded algebra, not just its homology ring.

We will assume that
our ground ring $k$ is an integral domain to simplify the proofs -- in
this paper we will only use the calculations for $k=\fp$.

 \medskip
 We will use the notation $\Lambda(y)=k[y]/y^2$ for the exterior
 algebra on $y$ over $k$, and $\Gamma(y)$ for the divided power
 algebra on $y$ over $k$, spanned over $k$ by elements
 $\gamma_{i}(y)$, $i\geq 0$, with  $\gamma_{i}(y)
 \cdot\gamma_{j}(y)=\binom{i+j}{i} \gamma_{i+j}(y)$.
\begin{prop}\label{O}
Let $k$ be an integral domain, and let $x$ be of even degree. Then
there exist quasi-isomorphisms
$$\pi\colon \B(k,k[x],k)\ra \Lambda(\epsilon x)$$
and
$$\inc\colon \Lambda(\epsilon x) \ra \B(k,k[x],k) $$
which are maps of differential graded  augmented commutative
$k$-algebras, with $ |\epsilon x| = |x|+1$.
\end{prop}

\begin{proof}
We define the quasi-isomorphisms as follows:  Let $\pi:\B (k,k[x],k)
\to \Lambda(\epsilon x)$ be given by $ \pi(1\otimes 1) =1$, 
$$\pi(1\otimes x^i\otimes 1) =
\begin{cases}
\epsilon x&\mbox{if } i=1, \\
0&\mbox{otherwise.}
\end{cases}
$$
and $\pi=0$ on $\B_n(k,k[x],k)$ for $n>1$.  Let
$\inc: \Lambda(\epsilon x) \to \B (k,k[x],k) $ be given by
$\inc(1)=1\otimes 1$ and $\inc(\epsilon x)=1\otimes x\otimes 1$.  Then
$\pi$ and $\inc$ 
are chain maps, and $\pi\circ\inc=\id_{ \Lambda(\epsilon x)}$.
Therefore $\inc_*$ induces an isomorphism from  $\Lambda(\epsilon x)$
to a direct summand of $H_*(\B (k,k[x],k))=\Tor_*^{k[x]}(k,k)$, and
$\pi_*$ projects back onto that summand.  But the resolution
$$0\ra \Sigma^{|x|}k[x] \xra{\cdotp x} k[x]$$
of $k$ shows that the rank  of $\Tor_*^{k[x]}(k,k)$ over $k$ in each 
degree is equal to that of $\Lambda(\epsilon x)$, and since $k$ is an
integral domain, the direct summand must then be equal to all of
$H_*(\B (k,k[x],k))$.  Thus $\pi$ and $\inc$ are quasi-isomorphisms.
In this case, both maps preserve the multiplication because both
$\B(k,k[x],k) $ and $\Lambda(\epsilon x)$ are graded commutative, so
the square of anything in odd degree must be zero.
\end{proof}

\begin{prop}\label{truncatedeven}
Let $k$ be an integral domain, let $m\geq 2$ be an integer, and let
$x$ be of even degree. Then there exist quasi-isomorphisms
$$\pi\colon \B(k,k[x]/x^m,k) \ra  \Lambda(\epsilon x) \otimes
\G(\varphi^0 x)$$
and
$$\inc\colon \Lambda(\epsilon x) \otimes
\G(\varphi^0 x) \ra \B(k,k[x]/x^m,k)$$
which are maps of differential graded augmented commutative 
$k$-algebras, with $ |\epsilon x| = |x|+1$ and $|\varphi^0 x|=2+ m |x|$.
\end{prop}

\begin{proof}
Let $\pi:\B(k,k[x]/x^m,k)\to
\Lambda(\epsilon x) \otimes \G(\varphi^0 x)$ be given by

$$\pi(1\otimes x^{a_1}  \otimes\cdots \otimes x^{a_n}     \otimes 1) =
\begin{cases}
x^{a_1+a_2-m}\cdots  x^{a_{n-1}+a_n-m}  \ \gamma_{({{n}\over 2})}
(\varphi^0 x) &n\ \mbox{even} , \\
x^{a_1-1}x^{a_2+a_3-m}\cdots  x^{a_{n-1}+a_n-m}\epsilon x \cdot
\gamma_{({{n-1}\over 2})} (\varphi^0 x)&n\ \mbox{odd,}
\end{cases}
$$
where $0\leq a_i<m$ and where we interpret $x^s=0$ for $s\neq 0$: for
$s<0$, this is because we define it to be so; for $s>0$, this is
because $k[x]/x^m$ acts by first applying the augmentation.
Therefore, if $n$ is
even,  we get $\gamma_{({{n}\over 2})} (\varphi^0 x) $ if and only if
$a_1+a_2=m, \ a_3+a_4=m, \ \ldots,\ a_{n-1}+a_n=m$ and otherwise we
get zero. For odd $n$  we get $\epsilon x \cdot
\gamma_{({{n-1}\over 2})}  (\varphi^0 x)$ if and only if  $a_1=1, \
a_2+a_3=m, \ \ldots,\ a_{n-1}+a_n=m$ and zero otherwise.
To see that $\pi$ is a chain map, we only need to show that it sends
boundaries to zero, which can be checked directly using the stringent
conditions under which a monomial is sent to a nonzero element.

 Let $\inc:  \Lambda(\epsilon x) \otimes
\G(\varphi^0 x) \to \B(k,k[x]/x^m,k)$ be given by
$$\inc (\gamma_{i} (\varphi^0 x) ) =1\otimes (x^{m-1}\otimes
x)^{\otimes i}\otimes 1\in \B_{2i}(k,k[x]/x^m,k)$$
and
$$\inc ( \epsilon x \cdot \gamma_{i} (\varphi^0 x) ) =1\otimes
x\otimes (x^{m-1}\otimes x)^{\otimes i}\otimes 1\in
\B_{2i+1}(k,k[x]/x^m,k).$$
Since $x^m=0$ and since the augmentation sends $x$ to zero, every face
map $d_j$ vanishes on the image of $\inc$, so clearly the boundary
vanishes too and $\inc$ is a chain map.

As before, we get that $\pi\circ\inc=\id_{ \Lambda(\epsilon x)\otimes
\G(\varphi^0 x)}$, and since the periodic resolution
 $$\ldots \ra  \Sigma^{(m+1)|x|} k[x]/{x^m} \xra{\cdotp x}
\Sigma^{m|x|} k[x]/{x^m} \xra{\cdotp x^{m-1}}
 \Sigma^{|x|} k[x]/{x^m}\xra{\cdotp x} k[x]/{x^m}$$
shows that $ \Lambda(\epsilon x)\otimes \G(\varphi^0 x)$ has the same
rank over $k$ in each dimension as
$H_*(\B(k,k[x]/x^m,k))=\Tor_*^{k[x]/x^m}(k,k)$, by the same argument
as in 
Proposition \ref{O}, $\pi$ and $\inc$ are quasi-isomorphisms.

To show that $\pi$ is multiplicative, consider
$\pi((1\otimes x^{a_1}  \otimes\cdots \otimes x^{a_\ell}     \otimes 1) \cdot
(1\otimes x^{a_{\ell+1}}  \otimes\cdots \otimes x^{a_{\ell+n}}     \otimes 1))$
which is the sum over all $(\ell,n)$-shuffles $\sigma$ of
$$\sgn (\sigma) \pi(1\otimes x^{a_{\sigma(1)} }  \otimes\cdots
\otimes x^{a_{\sigma(\ell+n)}}     \otimes 1).$$
In the case where $\ell$ and $n$ are both even, observe that this term
is equal to
$\sgn (\sigma)\gamma_ {({{\ell+n}\over 2})}( \varphi^0
x)$ if and only if $a_{\sigma(1)}+
a_{\sigma(2)}=m,\ldots,\  a_{\sigma(\ell+n-1)}+ a_{\sigma(\ell+n)}=m$.
If there is some pair $2i-1$, $2i$ for which $\sigma(2i-1)$ is in one
of the sets $\{1,\ldots,\ell\}$, $\{\ell+1,\ldots,\ell+n\}$ and
$\sigma(2i)$ is in the other, the term associated to $\sigma$ will
cancel with the term associated to the permutation which is exactly
like $\sigma$ except for switching  $\sigma(2i-1)$ and $\sigma(2i)$.
Thus we will be left with terms associated with shuffles $\sigma$
which shuffle pairs of coordinates, and for these it is clear that
$\pi(1\otimes x^{a_{\sigma(1)} }  \otimes\cdots \otimes
x^{a_{\sigma(\ell+n)}}     \otimes 1)\neq 0$ if and only if both
$\pi(1\otimes x^{a_1}  \otimes\cdots \otimes x^{a_\ell}     \otimes 1)\neq 0$ and
$\pi(1\otimes x^{a_{\ell+1}}  \otimes\cdots \otimes x^{a_{\ell+n}}
\otimes 1)\neq 0$.  And there will be exactly $\binom{{{\ell+n}\over
    2}} {{{\ell}\over 2}}$  $(\ell,n)$-shuffles $\sigma$ with 
$\sigma(2i)=\sigma(2i-1)+1$ for all $i$.

A similar argument works if $\ell$ is odd and $n$ is even.  Then the terms
corresponding to shuffles $\sigma$ which do not satisfy $\sigma(1)=1$
and $\sigma(2i+1)=\sigma(2i)+1$ for all $1\leq i < (\ell +n) /2$ will
cancel in pairs, and the terms corresponding to the
$\binom{{{\ell+n-1}\over 2}} {{{\ell-1}\over 2}}$
shuffles which do will be nonzero if and only  if the images of both
factors will be nonzero.  Commutativity then establishes
multiplicativity for the case $\ell$ even, $n$ odd.  If  both $\ell$
and $n$ are odd then all $(\ell,n)$-shuffles $\sigma$  will have a mixed
pair $2i-1$, $2i$ for which $\sigma(2i-1)$ is in one
of the sets $\{1,\ldots,\ell\}$, $\{\ell+1,\ldots,\ell+n\}$ and
$\sigma(2i)$ is in the other, so all the terms will cancel and so the product
will map to zero, which is also the product of the images of the factors.

\smallskip
To show that $\inc$ is multiplicative, it suffices to show that
$\inc( \epsilon x)\cdot \inc (\gamma_{i} (\varphi^0 x)) =\inc (
\epsilon x \cdot \gamma_{i} (\varphi^0 x) )$  and that
$\inc( \gamma_{i} (\varphi^0 x) )\cdot \inc (\gamma_{j} (\varphi^0 x)
) =\inc ( \gamma_{i} (\varphi^0 x)  \cdot \gamma_{j} (\varphi^0 x))=
\binom{i+j}{i} \inc (\gamma_{i+j} (\varphi^0 x))$. The first claim
follows from the fact that shuffles which allow two adjacent
$x$'s from different factors cancel in pairs, leaving only the unique
$(1, 2i)$-shuffle $\sigma$ with $\sigma(1)=1$. The second claim  follows
from the fact that shuffles which do not preserve the pairs
$x^{m-1}\otimes x$ cancel in pairs, and there are $\binom{i+j}{i} $
shuffles which preserve the pairs.
Thus both quasi-isomorphisms respect the multiplication.

\end{proof}

\begin{prop}\label{truncatedodd}
Let $k$ be an integral domain, let $x$ be of odd degree and let
$\rho^0x$ be an element with $ |\rho^0 x| = |x|+1$. Then
there exist quasi-isomorphisms
$$\pi \colon \B(k,\Lambda(x),k)\ra   \G(\rho^0 x)$$ 
and
$$\inc\colon  \G(\rho^0 x) \ra \B(k,\Lambda(x),k)$$
which are maps of differential graded augmented commutative 
$k$-algebras.
\end{prop}
If $k=\F_2$, this proposition and its proof also work if $x$ has even
degree, and the result agrees with the result of Proposition
\ref{truncatedeven} for $m=2$.
\begin{proof}
We use the same quasi-isomorphisms as in Proposition
\ref{truncatedeven}, and the argument showing that they are
quasi-isomorphisms is the same as well, but the multiplicative
structure is different  and much easier to analyze.  The maps from
Proposition \ref{truncatedeven} give, in the case of $m=2$,
$$\pi(1\otimes x^{a_1}  \otimes\cdots \otimes x^{a_n}     \otimes 1) =
\begin{cases}
\gamma_{n} (\rho^0 x)&\mbox{if }\ a_i=1\ \mbox{for\ all\ }1\leq i\leq n , \\
0&\mbox{otherwise}
\end{cases}
$$
and
$$\inc( \gamma_{n} (\rho^0 x) )=1\otimes x^{\otimes n}\otimes 1.$$
Since $x$ is of odd degree,
$$(1\otimes x^{\otimes i}\otimes 1)   \cdot
(1\otimes x^{\otimes j}\otimes 1) =
\binom{i+j}{i} (1\otimes x^{\otimes(i+ j)}\otimes 1)$$
for all $i,j\geq 0$ and so both $\pi$ and $\inc$ respect multiplication.
\end{proof}

\begin{notation}
\begin{enumerate}
\item[]
\item
If $k=\fp$, we can decompose the divided power algebra as
$$\G(\rho^0 x) \cong \bigotimes_{i
\geq 0} \fp[\gamma_{p^i
}(\rho^0
x)]/{({\gamma_{p^i
}(\rho^0 x)})^p}$$
and we will denote the generators $\gamma_{p^i
}(\rho^0 x)$ by $\rho^i
 x$.
\item
Similarly, if $k=\fp$
$$\G(\varphi^0 x) \cong \bigotimes_{i
\geq 0}
\fp[\gamma_{p^i
}(\varphi^0 x)]/{({\gamma_{p^i
}(\varphi^0 x)})^p}$$
and $\varphi^i
x$ is short for the generator $\gamma_{p^i
}(\varphi^0x)$ of the $i
$th truncated polynomial
algebra.
\end{enumerate}
\end{notation}

\section{Veen's spectral sequence and iterated tors} \label{sec:Veen}

Our main computational tool is the bar spectral sequence, set up in
\cite{V}, which is closely related to the bar constructions we use in
Section \ref{sec:hhh} and calculate the homology of in Section
\ref{sec:lemmas}.
Let $H\F_p$ denote  the Eilenberg-MacLane spectrum of $\F_p$. Veen uses the
Brun-Carlsson-Dundas \cite{bcd} model $\Lambda_{\mathbb{S}^n}H\fp$ for topological
Hochschild homology of order $n$ of $H\F_p$, $\THH^{[n]}(\F_p) = H\F_p
\otimes \mathbb{S}^n$.

\begin{thm} \cite[\S 7]{V} \label{thm:Veen}
There exists a strongly convergent spectral sequence of
$\fp$-Hopf algebras
$$E^2_{r,s}=\Tor_{r,s} ^{\pi_*(\Lambda_{\mathbb{S}^{n-1}}H\fp)}(\fp, \fp)
\Longrightarrow \pi_{r+s}(\Lambda_{\mathbb{S}^n}H\fp).$$
\end{thm}

Thus this spectral sequence uses
$\THH^{[n-1]}_*(\F_p)$ as an input in order to calculate
$\THH^{[n]}_*(\F_p)$.  As long as it keeps collapsing at $E^2$,
calculating $\THH_* ^{[n]}(\F_p)$ is simply a process of starting with
 $\THH_* (\F_p) =  \THH_* ^{[1]}(\F_p)\cong \fp[\mu]$ with $|\mu |=2$
 (as calculated by B\"okstedt in \cite{B}) and applying $\Tor_*^-(\fp,
 \fp)$ iteratively $n-1$ times.

By \cite[Theorem 7.6]{V}, this is what happens for  $ n\leq 2p$, and
so $ \THH^{[n]}(\fp)\cong B_n$ for $n\leq 2p$, where $B_n=\Tor^{B_{n-1}}(\F_p, \F_p)$ is the
iterated $\Tor$ ring as explained above and defined in Definition
\ref{defn:Bn} below.
We will actually show in Section \ref{sec:pushing} that
$\THH^{[n]}_*(\fp)\cong B_n$ up to 
$n \leq 2p+2$.  We believe that it should be possible to use spectrum
analogs of the methods of Section \ref{sec:lemmas} in order to understand the
homotopy type of the iterated $\Tor$ spectra rather than just their
homotopy rings, and prove that   $\THH^{[n]}_*(\fp)\cong B_n$ for all
$n>0$, and are working on showing that with Maria Basterra and Michael
Mandell.

It is well-known and follows from the calculations of  Section
\ref{sec:lemmas} that $\Tor_*^{\fp[x]}(\fp, \fp)\cong \Lambda(\epsilon
x)$ with $|\epsilon x|=1+|x|$, which would be odd if $|x|$ were even; 
that
$\Tor_*^{\Lambda[y]}(\fp, \fp)\cong \Gamma(\rho^0 y)$
if $|y|$ is odd, with $|\rho^0y| =|y|+1$, and that
 $\Tor^{\fp[z]/z^m}(\fp, \fp)=\Lambda(\epsilon z) \otimes \Gamma(\varphi^0 z)$
  when $|z|$ is even, with $|\epsilon z|=|z|+1$ and $|\varphi^0
  z|=2+m|z|$.  The latter includes the case 
$\Tor_*^{\Lambda[y]}(\fp, \fp)$
 if $|y|$ is even, as well as the case of  $\Tor_*^{\Gamma(y)}(\fp,
 \fp)$ for $|y|$ even, since the ground ring is $\fp$ and so 
$\Gamma(y)\cong \bigotimes_{k\geq 0} \fp [\varphi^k(y)]/ (\varphi^k(y))^p$.

One can prove that the $\mathrm{Tor}$ over a finite tensor product is
the tensor product of the $\mathrm{Tor}$'s
directly, using projective
resolutions of the single factors and the fact that $\F_p$ is
$\F_p$-flat.  Calculating $\mathrm{Tor}$ with the two-sided bar
resolution shows that $\mathrm{Tor}$ respects direct limits also in
the ring variable as well as in the module variables.

So we can encode the result of taking iterated  $\Tor_*^-(\fp, \fp)$
in a flowchart as in Figure \ref{figure1}.
\begin{figure}
$$ \xymatrix@C=-20pt{
{} & {} & {} & {} \\
{} & {\bigotimes_{k\geq 0} \Lambda(\epsilon \rho^k \epsilon \omega)} & {} & {} \\
{\F_p[\omega] \ra \Lambda(\epsilon \omega) \ra \G(\rho^0\epsilon
  \omega) \cong \bigotimes_{k\geq 0} \F_p[\rho^k \epsilon \omega]/(\rho^k
  \epsilon \omega)^p} \ar[ur] \ar[dr] & {\ldots} & {} & {} \\
{} & {\bigotimes_{k\geq
       0} \G(\varphi^0 \rho^k \epsilon \omega) \cong \bigotimes_{k,i\geq 0}
     \F_p[\varphi^i\rho^k\epsilon \omega]/(\varphi^i\rho^k\epsilon \omega)^p}
   & {} & {} \\ 
{} & {} & {} & {}
}$$
\caption{\label{figure1} Evolution of elements.}
\end{figure}
or more schematically as in Figure \ref{figure2}.
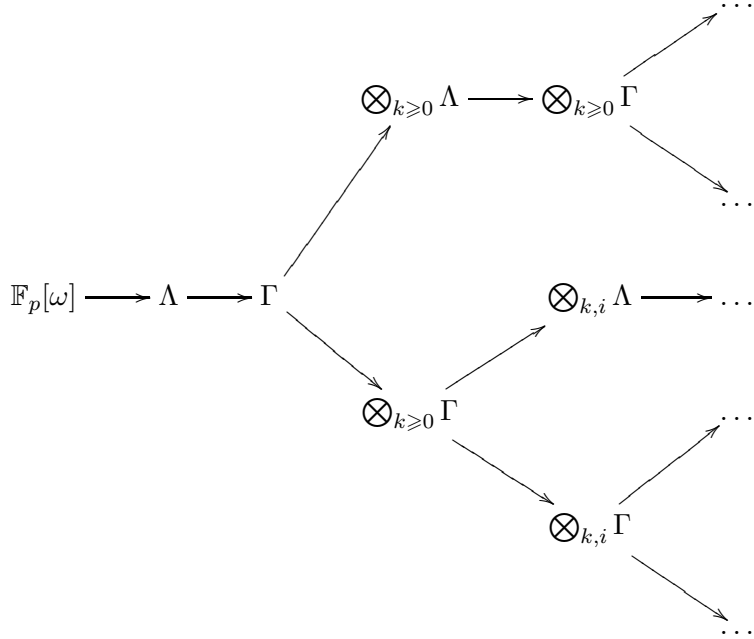
\begin{figure}
$$ \xymatrix{
{} & {} & {} & {} & {} & {\ldots}\\
{} & {} & {} & {\bigotimes_{k\geq 0} \Lambda}  \ar[r] &
{\bigotimes_{k\geq 0} \G}  \ar[ur] \ar[dr] & {}\\
{} & {} & {} & {} &
{} & {\ldots} \\
{\F_p[\omega]} \ar[r] & {\Lambda} \ar[r] & {\G} \ar[uur] \ar[dr] & {}
& {\bigotimes_{k,i} \Lambda} \ar[r] & {\ldots} \\
{} & {} & {} &  {\bigotimes_{k\geq
       0} \G} \ar[ur] \ar[dr]& {} & {\ldots} \\
{} & {} & {} & {} &  {\bigotimes_{k,i} \G} \ar[ur] \ar[dr]& {} \\
{} & {} & {} & {} & {} & {\ldots}\\
}$$
\caption{\label{figure2} Schematic overview of iterated Tor-terms.}
\end{figure}
This notation for elements in iterated Tor-terms goes back to Cartan
(compare \cite[\S 1]{C}). 

\begin{defn}\label{defn:Bn}
Let $B_n$ be the algebra generated by all words of length $n$ of the
following form (as illustrated in the flowchart), modulo the relations
implied in the description of the algebras above (free for $\mu$,
exterior for $\epsilon\omega$, polynomial truncated at the $p$th
power for $\rho^k\omega$ or $\varphi^k\omega$ for $k\geq 0$ and any
word $\omega$ of length $n-1$):
\begin{itemize}
\item The rightmost letter must be $\mu$.
\item If there is something to the left of $\mu$, it must be $\epsilon$.
\item If there is something to the left of an $\epsilon$, it must be a
  $\rho^k$ for some $k\geq 0$. 
\item If there is something to the left of a $\rho^k$ for any $k\geq
  0$, it must be either an $\epsilon$ or a $\varphi^j$ for some $j\geq
  0$. 
\item Similarly, if there is something to the left of a $\varphi^k$
  for any $k\geq 0$, it must be either an $\epsilon$ or a $\varphi^j$
  for some $j\geq 0$. 
\end{itemize}
Observe, by the discussion above, that $B_n$ is the algebra we get if
we apply the functor $\Tor_*^-(\fp, 
 \fp)$ iteratively $n-1$ times, starting with the algebra $\F_p[\mu]$.

\end{defn}

 \begin{defn}\label{defn:Bnprime}
Let $B'_n$ be defined as the
algebra generated by all words of length $n$ defined as above, except
that the rightmost letter must be $x$ rather than $\mu$; the letter
directly to its left, if there is one, should be an $\epsilon$.  This
follows the rules of the flowchart t
with $\omega=x$, and will
be useful in calculating $\HH^{[n]}_*(\fp[x])$.
\end{defn}

 \begin{defn}\label{defn:Bnprimeprime}
Let $B''_n=B''_n(m)$ be defined as the algebra generated by all words of length
$n$ ending with $\omega=x$ modulo the same relations as before and also the relation
$x^m=0$.  In this case, if there is a letter immediately to the left of
$x$, it has  to be either $\epsilon$ or $\varphi^k$ for some $k\geq 0$. The other rules are unchanged.  This will
be used in calculating $\HH^{[n]}_*(\fp[x]/x^m)$. As the $m$ should
usually be clear from the context, we will omit it from the notation.
\end{defn}

When we write such a word in an iterated Tor-term, the leftmost letter
in the word carries the 
information about what kind of algebra the element corresponding to
that word generates, the one before the last letter remembers what
kind of algebra the generator came from, and so on; exponents
remember what component of a divided power algebra the word came from
at a particular stage.

The bidegrees of the words are computed using the following recursive
formulas:

\begin{itemize}
   \item $|\mu |=2$ for the $\THH$ calculation, and  $|x|=0$ for the
     $\HH$ calculation, as explained above,
   \item $||\epsilon w||=(1, |w|)$,
   \item $||\rho^i w||=p^i (1, |w|)$, and
   \item $||\varphi^\ell w||=p^\ell (2, p|w|)$.
\end{itemize}
The bidegrees will be important in the $\THH$ calculation.  Note that
when we write $|w|$ on the right hand side of the formulas, we mean
the total degree of $w$.  For the $\HH$ calculations, we will only care about
total degrees.

\section{Pushing Veen's bounds}\label{sec:pushing}
In this section, we will work over $\F_p$ and assume  that $p>2$.  In
a Hopf algebra,
$\psi$ will denote the comultiplication.  
The following is a trivial generalization of \cite[Proposition
 4.1]{V}, adapted to the needs of our calculation.
It provides a little bit more information about the first nontrivial
differential one could have in Veen's spectral sequence. 

\begin{lem} \label{lem:supdif}
  Suppose that Veen's spectral sequence of Theorem \ref{thm:Veen}
  collapses at $E^2$  and has no nontrivial multiplicative extensions
  for all $i<n$, so that $\pi_*(\Lambda_{\mathbb{S}^{n-1}}H\fp)\cong
  B_{n-1}$.  Suppose also that in Veen's spectral sequence for
  $\pi_*(\Lambda_{\mathbb{S}^{n}}H\fp)$, 
  $d^j\equiv 0$ for all $2\leq j< i$.  If $d^i \not\equiv 0$, then
  there exists a generator $\gamma_{p^k}(x)$ in the $E^2=E^i$ term such
  that 
  $d^i(\gamma_{p^k}(x))$ is a nonzero linear combination of generators of
  exterior  algebras. 
\end{lem}

\begin{proof}
 If $d^i\not\equiv 0$, there exists an $a\in E^i_{*,*}$ such that
 $d^i(a) \neq 0$. Choose such an $a$ of lowest degree.  Recall that
 $E^i_{*,*}$ is a tensor product of graded exterior algebras and
 graded divided power algebras.  Writing $a$ as a linear 
  combination of pure tensors,
  we see that there must be a pure tensor $b$ such that $d^i(b) \neq
  0$.  If we can write $b = b'b''$ (with
  $b',b''$ of strictly lower 
  degree), then by the Leibniz rule, $d^i(b)=d^i(b') b'' \pm b'
  d^i(b'')$; by our assumption on the minimality of $b$'s degree, this
  sum must be zero, contradicting the fact that  $d^i(b) \neq 
  0$.  Thus $b$ must be indecomposable, that is: it must be 
a constant multiple of a generator.  If the bidegree of $b$ is
$(k,\ell)$, then the bidegree of $d^i(b)$ must be $(k-i, \ell+i -1)$,
and for $d^i(b)$ to be nontrivial, we must have $k\geq i\geq 2$.
Since all generators of an exterior algebra have bidegree $(1,\ell)$
for some $\ell$, we see that $b$ must be 
of the form $\gamma_{p^k}(x)$ for some $x$, and of even degree. 
  
  Now consider $d^i(b)$.  It must be primitive:
writing 
$\psi(b) = 1\otimes b + b\otimes 1 + \sum_j b'_j\otimes b''_j$, with
$b'_j$ and $b''_j$ of lower degree, we obtain that 
  \[\psi(d^i(b)) = 1\otimes d^i(b) + d^i(b)\otimes 1 +\sum_j
  (d^i(b'_j)\otimes b'' _j\pm 
b'_j\otimes d^i(b''_j))= 1\otimes d^i(b) + d^i(b)\otimes 1.\]  
The only primitive elements of odd degree in $E^i_{*,*}$ are generators of
exterior algebras. 
\end{proof}

Our goal is to show that Veen's bound of $n=2p$ can be pushed
 to $n=2p+2$ by a further  analysis of bi-degrees and the Hopf algebra
 structure, but no further: at 
$n=2p+3$ there will always be a differential candidate, which we
believe will in fact vanish, 
but that needs to be established by other methods.

\begin{defn} 
\begin{itemize}
\item[]
\item
Let $\#w$ denote the length of a word $w$, that is: the 
number of letters used to write $w$.  
\item
For a word $w$ we write $w^{[n]}$ for the word consisting of $w$
concatenated $n$ times.
\end{itemize}
 \end{defn}

\begin{lem} \label{lem:powerwords}
  The only word $w$ with $\#w \leq 2p+1$ and $|w| = 4p^k$ for $k\geq 0$ is equal
  to $\rho^k\epsilon\mu$.
\end{lem}

\begin{proof}
  Since the total degree $|w|$ is even, $w$ must start with a $\rho^\ell$ or a $\varphi^\ell$.  Suppose first that $w = \rho^\ell \epsilon w'$.  If $\ell < k$ then
  $|w'| = 4p^{k-\ell}-2$,   so by \cite[Lemma
 7.2 part 5]{V} we know that $w'$ equals
  $(\rho^0\epsilon)^{[p-2]}\mu$ or starts with $(\rho^0\epsilon)^{[p-2]}\varphi^0$
  or $(\rho^0\epsilon)^{[p-1]}$.  In the first case $|w'|=2p-2$,
  which is not of the form $4p^{k-\ell}-2$.
In the second  case, the beginning of $w'$ is of length $2p-3$, but it requires a tail of length $3$ or more, and thus $\# w' \geq 2p$, which is not possible.  In the third case, the beginning is of length $2p-2$, and so the only way we could get $\# w' =2p-1$ is by having $w'=(\rho^0\epsilon)^{[p-1]}\mu$, but then $|w'|= 2p\neq 4p^{k-1}$ and this case is also impossible.

Thus $\ell = k$, so that $w'=\mu$ and $w=\rho^k \epsilon\mu$.

  Now suppose that $w = \varphi^\ell w'$.  Then $p|w'| = 4p^{k-\ell} - 2$.
  However, this can only happen when $p=2$, a contradiction.  So there are no
  such possible words $w$, and we are done.
\end{proof}

We have the following extension of Veen's Theorem
 7.6:

\begin{prop} \label{prop:pushing}
  When $n\leq 2p+2$ there are no non-trivial differentials in the
  spectral sequence of Theorem
 \ref{thm:Veen}, and there is an $\F_p$-Hopf algebra isomorphism
  $\pi_*(\Lambda_{\mathbb{S}^n}H\F_p) \cong B_n$.
\end{prop}

\begin{proof}
  For $n\leq 2p$, \cite[Theorem
 7.6]{V}  gives
  us exactly the desired   result.  Thus we simply need to analyze two cases:
  $n=2p+1$ and $n=2p+2$.  In
  order to extend Veen's argument to these cases, we will need to show that
  \begin{enumerate}
  \item there are no possible non-trivial differentials in the
    spectral sequence, and
  \item there are no possible multiplicative extensions.
  \end{enumerate}

  \begin{enumerate}
  \item Suppose that there exists a possible nonzero differential.  This means
    that there exists an indecomposable element $\alpha$ and a primitive element
    $\beta$ with $|\alpha| = |\beta|+1$; as discussed in Lemma \ref{lem:supdif}
    we can assume that $\alpha$ is of the form $\gamma_{p^k}(x)$, or in other words
    that it is of the form $\rho^k w$  or $\varphi^k w$ for some
    admissible word $w$ of length $2p$ or $2p+1$, respectively.  In order for there to be a
    differential which might 
    not be trivial 
    on $\alpha$, we must have $k\geq 1$, so
    $|\alpha|\equiv 0 \pmod{2p}$.

    Then $|\beta| \equiv -1\pmod{2p}$.  As $\beta$ is primitive it is a linear
    combination of words that start with $\epsilon$.  From \cite[Lemma
 7.2]{V} we
    know that a word with such a degree is either equal to $\epsilon
    (\rho^0\epsilon)^{[p-2]}\mu$ or starts with 
    $\epsilon(\rho^0\epsilon)^{[p-2]}\varphi^0$ or
    $\epsilon(\rho^0\epsilon)^{[p-1]}\rho^k$ or
   $\epsilon(\rho^0\epsilon)^{[p-1]}\varphi^k$  for some $k\geq 1$.
   The first of these has 
    length $2p-2$ so is not under consideration.  The second must end with a
    suffix which has length at least $3$, so we'll need to consider it in both
    cases.  The third and fourth possibilities must end with a suffix
    of length at least $2$, so we'll 
    only need to consider them in the $2p+2$ case.

    \caselist{1}{$n=2p+1$}. All words that can be the target of
    differentials must be of the 
    form
    \[\beta=\epsilon(\rho^0\epsilon)^{[p-2]}\varphi^0\rho^k\epsilon\mu
    \qquad k\geq 0.\] 
    This word has degree $4p^{k+1}+2p-1$.  Thus any possible
    differential comes from a word of degree $4p^{k+1}+2p$.  As $\alpha$ must
    start with a $\rho^k$ or a $\varphi^k$, we know that $\alpha$ must
    equal $\varphi^1w$, 
    where $\#w = 2p$ and $|w| = 4p^k$ or $\rho^1\epsilon w$, where $\#w = 2p-1$
    and $|w| = 4p^k$.  However, both of these cases are impossible by Lemma
    \ref{lem:powerwords}, so there are no possible differentials.

    \caselist{2}{$n=2p+2$}. We have two possible words that might
    be targets of differentials:
    \[\beta_1 &= \epsilon(\rho^0\epsilon)^{[p-2]}\varphi^0\varphi^k\rho^\ell
 \epsilon\mu, \\
    \beta_2 &= \epsilon(\rho^0\epsilon)^{[p-1]}\rho^{k+1} \epsilon \mu.\] In
    both cases, $k,\ell\geq 0$.  We have
    \[|\beta_1| = 4p^{k+\ell+2} + 2p^{k+1} + 2p-1 \qquad |\beta_2| = 4p^{k+1} +
    2p-1.\] Thus we have two possibilities for $\alpha$, with $|\alpha_1| =
    4p^{k+\ell+2}+2p^{k+1} + 2p$ and $|\alpha_2| = 4p^{k+1} + 2p$.  As
    $\alpha_2$ must start with a $\rho^k$ or a $\varphi^k$, $k\geq 1$,
    it must be 
    of the form $\rho^1\epsilon w$ or $\varphi^1 w$ for some $w$ of
    length $2p$ or $2p+1$, respectively, with $|w|=4p^k$ or $|w|=4p^{k-1}$. 
    But we know (by Lemma 
    \ref{lem:powerwords}) that this is impossible, so it remains to
    consider the first case, where  $\alpha_1$ must equal either $\rho^1\epsilon w$
    with $\#w = 2p$ and $|w| = 4p^{k+\ell+1} + 2p^k$ or $\varphi^1\rho^m
    \epsilon w$ with $\# w = 2p-1$ and $|w| = 4p^{k+\ell-m} + 2p^{k-m-1}-2$.

    \caselist{2a}{$\alpha_1 = \rho^1\epsilon w$.}  First, note that $w\neq
    \rho^a \epsilon w'$, because in this case $|w'| = 4p^{k+\ell-a+1} +
    2p^{k-a}-2$ and $\#w' = 2p-2$, and $|w'|$ is either equal to $4p^{k+\ell-a+1}$
    (which is a contradiction by Lemma 
    \ref{lem:powerwords} because $\#w'=2p-2 > p \geq 3$) or equivalent to
    $-2\bmod 2p$, which demands a word longer than $2p-2$.  Thus $w = \varphi^a
    w'$.  Then $p|w'| = 4p^{k+\ell-a+1} + 2p^{k-a}-2$, which means that $a = k$
    and $|w'| = 4p^{\ell+1}$.  But $\#w' = 2p-1 > 3$, a contradiction by Lemma
    \ref{lem:powerwords}, and so $w$ does not exist.
    
    \caselist{2b}{$\alpha_1 = \varphi^1\rho^m\epsilon w$.}  We know that $|w| =
 4p^{k+\ell-m} + 2p^{k-m-1}-2$. If $k=m+1$ then this is equal to
    $4p^{k+\ell-m} \geq 4p$, and by Lemma \ref{lem:powerwords} we know that no
    such $w$ exists.  If $k>m+1$ then $|w| \equiv -2 \pmod{2p}$ and we know by
    \cite[Lemma
 3.3.2 part 5]{V} that $w$ must start with
    $(\rho^0\epsilon)^{[p-2]}\varphi^0$ or
    $(\rho^0\epsilon)^{[p-1]}\rho^k$ or
    $(\rho^0\epsilon)^{[p-1]}\varphi^k$ for some $k\geq 1$.  However,
    there are no words of 
    length $2p-1$ that start with any of these prefixes, so $w$ cannot
    exist.

  \item
 To solve the multiplicative extension problem we need to determine what
    the $p$th powers of elements can be.  Let $z$ be a generator of lowest
    degree with $z^p\neq 0$.  Then we have
    \[\psi(z^p) = \psi(z)^p = 1\otimes z^p + z^p\otimes 1 + \sum (z')^p\otimes
 (z'')^p = 1\otimes z^p + z^p\otimes 1,\] so $z^p$ must be primitive.
    However, in addition we know that $|z^p| = p|z|$, so $|z^p| \equiv 0
\pmod{2p}$.  By the proof of \cite[Lemma 7.5]{V} the shortest primitive word
    with degree
    equivalent to $0$ modulo $2p$ of degree larger than $2p$ is equal to $w =
    (\rho^0\epsilon)^{[p-1]}\varphi^0\rho^k\epsilon\mu$ for $k\geq 1$.
    Thus it has 
    length $2p+2$, so we do not need to worry about multiplicative extensions in
    the $n=2p+1$ case.

    In the $n=2p+2$ case, we need some extra care.  The degree of $w$
    is $|w| =
4p^{k+1}+2p$, so we see that $|z| = 4p^k+2$.  Therefore $z =
    \rho^0\epsilon w$ or $z = \varphi^0\rho^\ell w$.  In the first
    case we 
    have $\#w=2p$ and $|w| = 4p^k$, so by Lemma \ref{lem:powerwords} this cannot
    happen.  In the second case, we can deduce $\#w = 2p$ and $|w| =
    4p^{k-\ell-1}-1$.  Note that we must have $k-\ell-1 > 0$, as otherwise this
    clearly cannot happen. But then we know that $|w| \equiv -1 \pmod{2p}$, and
    by \cite[Lemma
 7.5]{V} it must have length at least $2p+1$.
    Thus such a word
    does not exist, and we see that there are no multiplicative extensions when
    $n=2p+2$, either.
  \end{enumerate}
\end{proof}

As we mentioned above,  it is not possible to continue pushing the
bound using this type of 
analysis, and while the spectral sequence may continue to collapse for
$n>2p+2$ (as we believe it will) we cannot deduce this purely from degree
considerations: 

\begin{prop} \label{prop:n=2p+3}
  For $n= 2p+3$ there is a potential non-trivial differential.
\end{prop}

\begin{proof}
  Let
  \[w = \varphi^1(\rho^0\epsilon)^{[p-1]}\varphi^0\rho^0\epsilon\mu \mathand v =
  \epsilon(\rho^0\epsilon)^{[p-2]}\varphi^0\rho^2\epsilon \rho^0\epsilon\mu.\]
  We have
  \[ ||w||= (2p, 6p^3)  \mathand ||v|| = (1, 6p^3+2p-2).\]
  Thus we have a  differential $d^{2p-1}$ in
  the spectral sequence that is potentially non-trivial.

\end{proof}

\begin{rem}
  We do not claim that this is the shortest possible differential.  It may be
  that for more complicated words there exist shorter possible differentials;
  indeed, at $n=2p+4$ it is easy to find potential differentials of
  length $p-1$.
\end{rem}

We found the above potential differential using a computer program
written in Haskell; we 
include the code in Appendix \ref{app:code}.

\section{$\THH^{[n]}(\F_2)$, up to $n=3$ and a stable element}
Marcel B\"okstedt showed \cite{B} that $\THH$ of $\F_2$ is isomorphic to a
polynomial algebra on a generator in degree $2$, $\F_2[\mu]$. Using Torleif
Veen's  \cite{V} spectral sequence
$$E^2_{r,s}=\mathrm{Tor}_{r,s}^{\THH_*^{[n]}(\F_2)}(\F_2,\F_2)
\Rightarrow  \THH_{r+s}^{[n+1]}(\F_2)$$ 
we obtain
$$ \THH_*^{[2]}(\F_2) \cong \F_2[\beta]/\beta^2$$
where $\beta$ is a generator in degree three (see also
\cite[Proposition
 2.3.1]{Vthesis}).

Using Proposition \ref{truncatedodd} we get a spectral sequence
calculating $\THH_*^{[3]}(\F_2)$ with $E^2$-term
$$\mathrm{Tor}_{*,*}^{\THH_*^{[2]}(\F_2)}(\F_2,\F_2) \cong \bigotimes_{i\geq 0}
\F_2[\gamma_{2^i}(x)]/\gamma_{2^i}(x)^2,  \, \text{ with } |x|=4.$$
The generators are concentrated in bidegrees of the form $(k,3k)$ so
there are no non-trivial differentials and the  spectral sequence
collapses. Also, since the only possible products are those which are
detected by the $E^\infty$ term, there are no multiplicative extension
issues, so we get: 

\begin{prop} Let $x$ denote a generator in degree $4$, then
$$ \THH^{[3]}_*(\F_2) \cong \bigotimes_{i\geq 0}
\F_2[\gamma_{2^i}(x)]/\gamma_{2^i}(x)^2.$$
\end{prop}
As
$$\mathrm{Tor}^{\bigotimes_{i\geq 0}
\F_2[\gamma_{2^i}(x)]/\gamma_{2^i}(x)^2}(\F_2,\F_2) \cong \bigotimes_{i\geq 0}
\mathrm{Tor}^{\F_2[\gamma_{2^i}(x)]/\gamma_{2^i}(x)^2}(\F_2,\F_2)$$
we have to understand the single factors first.
For each factor of the tensor product, by Proposition \ref{truncatedeven}
$$ \mathrm{Tor}^{\F_2[\gamma_{2^i}(x)]/\gamma_{2^i}(x)^2}(\F_2,\F_2) \cong
\bigotimes_{j\geq 0} \F_2[\gamma_{2^j}(y_i)]/\gamma_{2^j}(y_i)^2 \cong
\Gamma_{\F_2}(y_i)$$
with the $y_i$'s being elements of bidegree $(1,2^{i+2})$. But the
$E^2$-term is now a tensor product of these building blocks
$$ E^2_{*,*} \cong \bigotimes_{i\geq 0} \Gamma_{\F_2}(y_i) \cong
 \bigotimes_{i\geq 0}\bigotimes_{j\geq 0}
\F_2[\gamma_{2^j}(y_i)]/\gamma_{2^j}(y_i)^2 $$ thus excluding non-trivial
 differentials is harder.

\begin{lem}
The elements in the first column of the
spectral sequence
$$E^2_{*,*} = \Tor^{\THH_*^{[3]}}(\F_2, \F_2) \Longrightarrow
\THH_*^{[4]}(\F_2, \F_2)$$
are not in the image of $d^r$ for any $r$.
\end{lem}
\begin{proof}
The spectral sequence is a bar spectral sequence and the filtration
that gives rise to it is compatible with the multiplication in the bar
construction. Therefore the spectral sequence is (at least) one of algebras. It
therefore suffices to show that none of the indecomposable elements can
hit anything in the first column.  Note that the only elements on the first column are the $y_i$'s.

The bidegree of an element $\gamma_{2^j}(y_i)$ is $(2^j,2^j\cdot 2^{i+2})$
and if a $d^r(\gamma_{2^j}(y_i))$ is in the first column for $r\geq 2$ then $r=2^j-1$ and the relation in the internal
degree forces $2^j(2^{i+1}+1)-2$ to be of the form $2^{k+2}$. Since
$r\geq 2$, we must have $j\geq 2$, but then
$2^j(2^{i+1}+1)-2= 2(2^{j-1}(2^{i+1}+1)-1)$ is not of the form
$2^{k+2}$.

So no indecomposable element hits anything in the
first column. Products of such elements cannot hit a $y_i$ either,
because this would decompose $y_i$ (the spot $(0,0)$ cannot be hit by
a differential for degree reasons),  so all the $y_i$ must survive to the $E^\infty$ term.
\end{proof}
\begin{rem}
Veen \cite[Proposition 3.5]{V} describes the stabilization map
$$ \sigma\colon \THH_*^{[n]}(R) \ra  \THH_{*+1}^{[n+1]}(R)$$ for every commutative
ring spectrum $R$. It sends a class $[z] \in  \THH_q^{[n]}(R)$ to the
element in $\THH_{q+1}^{[n+1]}(R)$ that corresponds to $1 \otimes [z]
\otimes 1 \in B_1(\pi_0(R), \THH_q^{[n]}(R), \pi_0(R))$. From the first cases we
can read off that $\sigma$ sends $\mu \in \THH_2^{[1]}(\F_2)$ to
$\beta \in \THH_3^{[2]}(\F_2)$ and $\beta$ to $x \in
\THH_4^{[3]}(\F_2)$. We know that the $y_i$'s give rise to non-trivial
elements in $\THH_{1+2^{i+2}}^{[4]}(\F_2)$ and that $\sigma(x) =
 y_0 \in \THH_{5}^{[4]}(\F_2)$.
\end{rem}
\begin{prop} \label{prop:stable}
The iterative classes $\sigma^i(y_0)$ are all non-trivial and therefore  give
rise to a non-trivial class in topological Andr\'e-Quillen homology,
$\TAQ$,
$$ \TAQ_1(\F_2) := \varinjlim_n \THH_{1+n}^{[n]}(\F_2).$$
\end{prop}
\begin{proof}
We know that the classes $\sigma^i(y_0)$ are always cycles in the
corresponding spectral sequences, so we have to show that they cannot
be hit by any differential.
We do not know whether the $\gamma_{2^j}(y_i)$'s survive but we know that
the $E^\infty$-term is a subquotient of the $E^2$-term and hence we
get at most elements in $\THH^{[4]}_*(\F_2)$ that have a total degree
corresponding to products of the $\gamma_{2^j}(y_i)$'s. By an iteration of
this argument we can calculate possible bidegrees of elements that would
arise if there were no non-trivial differentials. Let $\ell$ be bigger
or equal to two and consider elements
$\gamma_{2^{i_{\ell+1}}}(y_{i_1,\ldots,i_\ell})$ of
bidegree
$$(2^{i_{\ell+1}}, 2^{i_{\ell+1}}(2^{i_{\ell}} +
2^{i_{\ell}+i_{\ell-1}}+\ldots+2^{i_{\ell}+i_{\ell-1}+\ldots+i_2}+2^{i_{\ell}+
  i_{\ell-1}+ \ldots+i_2+i_1+2})).$$
A product of elements
$\gamma_{2^{i_{1, \ell+1}}}(y_{i_{1,1},\ldots,i_{1,\ell}})$ up to
$\gamma_{2^{i_{m, \ell+1}}}(y_{i_{m,1},\ldots,i_{m,\ell}})$
then has homological degree
$\sum_{j=1}^r 2^{i_{j, \ell+1}}$ and internal degree
$$ \sum_{j=1}^r 2^{i_{j,\ell+1}+i_{j,\ell}} + \ldots +
\sum_{j=1}^r 2^{i_{j,\ell+1}+i_{j,\ell}+\ldots+i_{j,2}} +
\sum_{j=1}^r
2^{i_{j,\ell+1}+i_{j, \ell}+\ldots+i_{j,2}+i_{j,1}+2}. $$

We know that $y_{0,\ldots,0} = \gamma_{2^0}(y_{0,\ldots,0})$ has bidegree
$(1,\ell-1+4)=(1,\ell+3)$. If a differential $d^s$ hits this element,
then it has to start in something of bidegree $(1+s,\ell+3-s+1)=
(s+1,\ell+4-s)$. For $s\geq 2$ the only possible bidegrees are
$(3,\ell+2)$ up to $(\ell+5,0)$.

The element $\gamma_2(y_{0,\ldots,0})$ has bidegree $(2,(\ell-1)2+8) =
(2,\ell + (\ell+6))$  and as $\ell$ is at least $2$ the internal
degree is already
larger than $\ell+2$, so this element cannot be a suitable
source for a nontrivial differential.  All other potential
bidgrees have larger internal degree, thus there are no non-trivial
differentials.

\end{proof}

Maria Basterra and Michael Mandell calculated $\TAQ_*(H\F_p)$ for every
prime $p$ (see \cite[\S 6]{lazarev} for a written account) and there
is precisely one generator in $\TAQ_1(H\F_p)$.

\begin{rem}
For odd primes $p$ it is easy to see that the generator $\mu \in
\THH_2(\F_p)$ stabilizes to a non-trivial class in
$\TAQ_1(H\F_p)$. The stabilizations of $\mu$ are represented by the words 
$((\rho^0\epsilon)^\ell\mu)$ and
$(\epsilon(\rho^0\epsilon)^\ell\mu)$ in the spectral sequences (for
some $\ell$), so 
we have to show that these elements cannot be hit by any 
differential. Both types of elements are of bidegree $(1,m)$ for some $m$. If
$d^r\colon E^r_{s,t} \ra E^r_{s-r,t+r-1}$ should hit an element in
such a spot, then we get $s=r+1$ and $t=m-r+1$. As $r$ is greater or equal to
$2$, the differential can only start from bidegrees of the form
$(3,m-1),\ldots,(m+2,0)$. If  a term arises in the same  spectral
sequence as a stabilization of $\mu$ with bidegree $(1,m)$, then it is
generated by words of  
length $m$, which means that it has internal degree at least $m$. But
such terms  cannot hit a term with bidegree
$(1,m)$, so the stabilizations of $\mu$ survive. 
\end{rem}

\section{A splitting of $\THH^{[n]}(A[G])$ for abelian groups $G$}
\label{sec:splitting}
If $G$ is an abelian group, then the suspension spectrum of $G_+$ is an
$E_\infty$ ring spectrum, so it can be made into a commutative
$S$-algebra $S^0[G]$ for instance by the methods of \cite{EKMM}.  If
$R$ is another 
commutative $S$-algebra, so is $R\wedge S^0[G]$.  Applying the
formula for the product of two simplicial objects, we get that for any
$n$ and any commutative $S$-algebras $A$ and $B$,
$$\THH^{[n]}(A\wedge B) \simeq \THH^{[n]}(A) \wedge  \THH^{[n]}(B),$$
which in our case yields
$$\THH^{[n]}(R\wedge S^0[G]) \simeq   \THH^{[n]}(R) \wedge  \THH^{[n]}( S^0[G]).$$
If $R$ is a general $S$-algebra, we could take $R\wedge S^0[G]$ with
coordinate-wise product to be the definition of $R[G]$.  If
$R=HA$ is the Eilenberg Mac Lane spectrum of a commutative ring, this
is a model of the Eilenberg Mac Lane spectrum $H(A[G])$.  This is
because $HA\wedge S^0[G]$ has only one nontrivial stable homotopy
group;  $HA\wedge S^0[G]$ is the coproduct in the category of
commutative $S$-algebras so the obvious inclusions
induce a map of commutative $S$-algebras $HA\wedge S^0[G] \to
H(A[G])$ which  induces a multiplicative isomorphism on that unique
nontrivial homotopy group. The product on an Eilenberg Mac Lane
spectrum is determined by what it does on the unique nontrivial
homotopy group,
so we get
\begin{equation} \label{eq:split}
\THH^{[n]}(A[G]) \simeq   \THH^{[n]}(A) \wedge  \THH^{[n]}( S^0[G]).
\end{equation}
As usual, when we talk of the topological Hochschild homology of a
ring, we mean the topological Hochschild homology of its Eilenberg Mac
Lane spectrum.

\begin{prop}
If $A$ is a commutative $\F_p$-algebra, then for any $n\geq 1$ and any
abelian group $G$,
$$\THH^{[n]}_*(A[G]) \cong \THH^{[n]}_* (A)\otimes \HH^{[n]}_* (\F_p[G]).$$
\end{prop}

\begin{proof}
We can rewrite the splitting in \eqref{eq:split} above as
$$\THH^{[n]}(A[G]) \simeq   \THH^{[n]}(A) \wedge_{H\F_p} H\F_p\wedge
\THH^{[n]}( S^0[G]),$$
which yields a spectral sequence with $E^2$-term 
$$\mathrm{Tor}_{*,*}^{\F_p}( \THH^{[n]}_* (A), \pi_* (H\F_p\wedge
\THH^{[n]}( S^0[G])) ) \cong
\THH^{[n]}_* (A) \otimes H_*( \THH^{[n]}( S^0[G]);\F_p)$$
converging to $\THH^{[n]}_*(A[G]) $.  (Recall that for a commutative
$\F_p$-algebra $A$, $ \THH^{[n]}(A) $ is an $HA$-module, and so its
homotopy groups are $\F_p$-vector spaces.)  Since the spectral
sequence is concentrated in the $0$th column, it collapses, yielding
$$ \THH^{[n]}_*(A[G]) \cong \THH^{[n]}_* (A) \otimes H_*( \THH^{[n]}(
S^0[G]);\F_p) \cong  \THH^{[n]}_* (A)\otimes \HH^{[n]}_* (\F_p[G]),$$
where the fact that $ H_*( \THH^{[n]}( S^0[G]);\F_p) \cong
\HH^{[n]}_* (\F_p[G])$ follows from the fact that $H_*(S^0[G]; \F_p)$
consists only of $\F_p[G]$ in dimension zero and the K\"unneth formula.
\end{proof}

Note that this proof goes through if we replace $G$ by any
commutative monoid $M$. 

\section{The higher B\"okstedt spectral sequence}

The aim of this section is to provide a B\"okstedt spectral sequence
for $\THH^{[n]}_*$.
\begin{notation} \label{notation:sphere}
For the remainder of the paper $\mathbb{S}^1$ will always denote
the standard model of the $1$-sphere  with two non-degenerate simplices,
one in dimension zero and one in dimension one.
For $n\geq 1$ we
take the $n$-fold smash product of this model as a simplicial model of
$\mathbb{S}^n$.
\end{notation}

Assume that $R$ is a cofibrant commutative $S$-algebra (in the setting of
\cite{EKMM}). Then the simplicial
spectrum $\THH^{[n]}(R)_\bullet$ has  $k$-simplices
$$\THH^{[n]}(R)_k = \bigwedge_{\mathbb{S}^n_k} R.$$

The inclusion from the `subspectrum' of degenerate
simplices into the simplicial spectrum  (which is actually a map of
co-ends, as in 
\cite[p.182]{EKMM}) is a cofibration, because  the degeneracies are
induced by the 
unit of the algebra and the fact that $R$ is cofibrant as a commutative
$S$-algebra \cite[VII Theorem
 6.7]{EKMM} guarantees that the smash product
has the correct homotopy type. Therefore the simplicial spectrum
$\THH^{[n]}(R)_\bullet$ is proper.  

By \cite[X 2.9]{EKMM} properness implies that there is a spectral
sequence for any homology theory $E$ with
$$E^2_{r,s} = H_r(E_s(\THH^{[n]}(R)_\bullet))$$
converging to $E_{r+s}\THH^{[n]}(R)$. 
Note that for every $s$, $E_s(\THH^{[n]}(R)_\bullet)$ is a simplicial
abelian group; 
$H_r(E_s(\THH^{[n]}(R)_\bullet))$ denotes its $r$'th homology group. 

In the following we identify the $E^2$-term in good cases. 

If $E_*(R)$ is flat over $E_*$, then we get that $E_s(\THH^{[n]}(R)_r)$
is
$$\pi_s(E\wedge_S \THH^{[n]}(R)_r) \cong \pi_s(E
\wedge \bigwedge_{\mathbb{S}^n_r} R) \cong
\pi_s(\bigwedge^E_{\mathbb{S}^n_r} E \wedge R) \cong
(\bigotimes^{E_*}_{\mathbb{S}^n_r} E_*(R))_s$$
where $\bigwedge^E$ indicates that the smash product is taken over $E$. Taking
the $r$th homology of the corresponding chain complex gives precisely
$$E^2_{r,s} \cong \HH^{[n]}_{r,s}(E_*(R))$$
 where $r$ is the homological degree and $s$ the internal one. 
Therefore the B\"okstedt spectral sequence for higher $\THH$ is of the
following form.  

\begin{prop}
Let $R$ be a cofibrant commutative $S$-algebra and let $E$ be a
homology theory such that $E_*(R)$ is flat over $E_*$. Then there is a
spectral sequence 
$$E^2_{r,s} \cong \HH^{[n]}_{r,s}(E_*(R)) \Rightarrow E_{r+s}(\THH^{[n]}(R)). $$
\end{prop}

For $E=H\mathbb{F}_p$ we get $\HH^{[n]}_{r,s}((H\F_p)_*(R))$ for instance. If
we then set $R=H\F_p$ as well, we obtain

$$E^2_{r,s} \cong \HH^{[n]}_{r,s}((H\F_p)_*(H\F_p))$$
thus we have to calculate Hochschild homology of order $n$ of the dual
of the mod-$p$ Steenrod algebra, $\mathcal{A}_*(p)$.

For $p=2$ this is a polynomial algebra in classes $\xi_i$ of degree
$2^i-1$ and for $i \geq 1$. We can write $\mathcal{A}_*(2)$ as
$$ \mathcal{A}_*(2) \cong \bigotimes_{i\geq 1} \F_2[\xi_i].$$

Recall that Pirashvili defines Hochschild homology of order $n$ of a
commutative $k$-algebra $A$
the homotopy groups of the Loday functor $\mathcal{L}(A,A)$ evaluated on a
simplicial model of $\mathbb{S}^n$ \cite[5.1]{P}. For a finite pointed
set of the form $\{0,\ldots,m\}$ with $0$ as basepoint
$\mathcal{L}(A,A)\{0,\ldots,m\}$ is $A^{\otimes m+1}$ and a map of
finite pointed sets $f\colon \{0,\ldots,m\} \ra \{0,\ldots,M\}$ induces
    a map of tensor powers by
$$ f_*(a_0\otimes \ldots\otimes a_m) = b_0\otimes \ldots\otimes b_M,
\, b_i=\prod_{f(j)=i}a_j$$
where the product over the empty set spits out the unit of the algebra
$A$. For a finite pointed simplicial set $X.$ the Loday functor on $X.$
is then defined to be the simplicial $k$-module with $m$-simplices
$$\mathcal{L}(A,A)(X.)_m = \mathcal{L}(A,A)(X_m).$$
Therefore, for any  two commutative
algebras $A, B$ we have
$$ \mathcal{L}(A\otimes B,A\otimes B) \cong  \mathcal{L}(A,A) \otimes
\mathcal{L}(B,B)$$
as functors and so
$$ \pi_*\mathcal{L}(A\otimes B,A\otimes B)(\mathbb{S}^n) \cong
\pi_*(\mathcal{L}(A,A)(\mathbb{S}^n) \otimes
 \mathcal{L}(B,B)(\mathbb{S}^n)).$$
If all the algebras involved are flat as $k$-modules, we can identify this with
$$ \pi_*(\mathcal{L}(A,A)(\mathbb{S}^n)) \otimes
\pi_*(\mathcal{L}(B,B)(\mathbb{S}^n)).$$

In our case, where we are working over $\fp$, we can therefore break
down B\"okstedt's spectral sequence
$\HH^{[n]}_{r,s}(\mathcal{A}_*(p))$ into a tensor product of the
higher Hochschild homology of the different tensored factors of 
$\mathcal{A}_*(p)$.

\smallskip

We know that
$$ \HH^{[n]}_*(k[x];k) \cong H_*(K(\mathbb{Z},n);k)$$
(see
for instance \cite[p.~207]{lr}).
Here $\HH^{[n]}_*(k[x];k)$ denotes Hochschild homology of order $n$ of
$k[x]$ with coefficients in $k$. So we
have to understand what difference an internal grading makes and what
changes if we take coefficients in $k[x]$ and not just in $k$.

\section{Higher Hochschild homology of (truncated) polynomial
  algebras}\label{sec:hhh} 
In this section we will explain how to compute the higher Hochschild
homology of the rings $k[x]$  over any integral domain $k$, and
$\fpxl$ over $\fp$.   By varying the ground ring over which the tensor
products in the Loday construction are taken, we can exhibit higher
Hochschild homology as iterated Hochschild homology.  Because we will
be varying the ground rings, we introduce the notation
$\mathcal{L}^k (R,M)$ to indicate the ground ring $k$ in the Loday
construction.

These methods were suggested to us by Michael Mandell based on his
work with Maria Basterra on  $\TAQ$ computations. Note that most of
this section involves formal constructions that 
could be applied to augmented commutative $H\fp$-algebra spectra as
well. 

\begin{lem} \label{lem:twoRs}
Let $k$ be a commutative ring, and let $R$ be a commutative
$k$-algebra.  Then there is an isomorphism of functors from pointed
simplicial sets to simplicial augmented  commutative  $R$-algebras
$$\mathcal{L}^k (R,R)\cong \mathcal{L}^R (R\otimes_k R,R),$$
where $R$ acts on  $R\otimes_k R$ by multiplying the first coordinate,
and the augmentation map is the multiplication $R\otimes_k R\to R$.
\end{lem}

\begin{proof}
We can define a natural transformation $\mathcal{L}^k (R,R)\to
\mathcal{L}^R (R\otimes_k R,R)$ by 
mapping $R\hookrightarrow R\otimes_k R$ via $r\mapsto 1\otimes r$ over
each simplex other than the base point, and using the identity over
the base point.  This map is simplicial, and is an isomorphism in each
simplicial degree.
\end{proof}

\begin{rem}\label{rem:HHandbar}
For any commutative ring $R$ and augmented commutative  $R$-algebra
$C$, there is an isomorphism of simplicial augmented commutative
$R$-algebras
$$\B^R(R,C,R)  \cong \mathcal{L}^R (C,R)(\mathbb{S}^1) ,$$
where $\B^R$ denotes the two-sided bar construction with tensors taken
over $R$ and $\mathbb{S}^1$ is the model of the $1$-sphere as in
\ref{notation:sphere}. This is simply because we can map the two $R$'s
on the sides of the
bar complex to the $0$th (coefficient) coordinate in the Hochschild
homology complex.
\end{rem}

\begin{lem} \label{lem:smashingspaces}
Let $R$ be a commutative ring, and let $C$ be an augmented commutative 
$R$-algebra.  Let $X.$ and $Y.$ be pointed simplicial sets.  Then
there is an isomorphism between the diagonals of the bisimplicial
augmented commutative $R$-algebras
$$\mathcal{L}^R (\mathcal{L}^R(C,R)(X.), R)(Y.) \cong \mathcal{L}^R
(C,R)(X.\wedge Y.)$$
\end{lem}

If $X.$ is a pointed simplicial set, then we denote
by $\tilde{X}_k$ the $k$-simplices of $X$ that are not the basepoint.

\begin{proof}
In degree $k$ we can identify  the diagonal of the bisimplicial sets as
$${\bigotimes_{\tilde Y_k}} ((\bigotimes_{\tilde X_k} C)\otimes R)
\otimes R \cong \bigotimes_{\tilde X_k\times \tilde Y_k} C \otimes R.$$
Here, tensor products are all taken over $R$. The 
non-basepoint $k$-simplices
in $X.\wedge Y.$ are exactly $\tilde X_k\times \tilde Y_k$, and the
simplicial face maps in both cases are induced from those of $X.$ and
$Y.$  in the same way.
\end{proof}

\begin{cor}\label{cor:iterated}
For any commutative ground ring $k$ and commutative $k$-algebra $R$,
the $n$th higher Hochschild homology complex of $R$ over $k$,
$\HH^{[n]}(R)$, 
can be written as
$$\HH^{[n]}(R)\cong \B^R(R,\HH^{[n-1]}(R),R).$$
\end{cor}

\begin{proof}
By Lemmata \ref{lem:twoRs} and \ref{lem:smashingspaces} and Remark
\ref{rem:HHandbar},
\begin{equation*}
\begin{split}
\HH^{[n]}(R) & = \mathcal{L}^k (R,R)(\mathbb{S}^n)  \cong  \mathcal{L}^R
(R\otimes_k R,R) (\mathbb{S}^n)\cong
\mathcal{L}^R (\mathcal{L}^R(R\otimes R,R)(\mathbb{S}^{n-1}), R)(\mathbb{S}^1)
\\ &
\cong \mathcal{L}^R (\HH^{[n-1]}(R), R)(\mathbb{S}^1)\cong
\B^R(R,\HH^{[n-1]}(R),R).
\end{split}
\end{equation*}
\end{proof}

\begin{rem}
Our results in Corollary \ref{cor:iterated} are not new.  They can be
found in the literature for slightly different settings: For instance,
Veen \cite{V} establishes such an
identification for ring spectra and the \cite{bcd}-model in order to
construct his spectral sequence and Ginot-Tradler-Zeinalian prove in
an $(\infty,1)$-category setting that
the Hochschild functor sends homotopy pushouts on space level to derived tensor
products \cite[3.27 c)]{gtz}.
\end{rem}
Now we can calculate $\HH^{[n]}(R) $ inductively.  To work with the
bar construction, observe first that if we calculate
$\B^R(R,C,R)$ for an augmented commutative  $R$-algebra $C$ and if there
is an augmented commutative  $k$-algebra $C'$ so that $C\cong R \otimes
C'$ as an augmented commutative $R\otimes k$-algebra (that is, the augmentation
$C\to R$ is the tensor product of the identity of $R$ with an
augmentation $C'\to k$), then by grouping the $R$'s together we get

$$\B^R(R,C,R)\cong \B^R(R,R,R) \otimes \B^k(k,C',k)\cong R  \otimes
\B^k(k,C',k)$$
as simplicial augmented commutative  $R\cong R\otimes k$-algebras.  Also, if we
have a tensor product of augmented commutative $k$-algebras $C$ and
$D$,
$$\B^k(k, C\otimes D, k)\cong \B^k(k, C, k) \otimes \B^k(k, D, k)$$
as simplicial augmented commutative  $k$-algebras.

\bigskip
In  \cite{B},  B\"okstedt used such decompositions to calculate the
Hochschild homology of the dual of the Steenrod algebra.   He observed
that for any commutative ring $k$,
$$k[x]\otimes k[x]\cong k[x]\otimes C',$$
as augmented commutative  algebras, where $k[x]$ is embedded as
 $k[x]\otimes k \subset k[x]\otimes k[x]$, and $C'\subset k[x]\otimes
k[x]$ is the sub-algebra generated over $k$ by the element
$x'=x\otimes 1 -1\otimes x$.  Note that $C'= k[x']\cong k[x]$.

\begin{thm}\label{thm:kofx}
Let $k$ be an integral domain.
There is an isomorphism of  simplicial augmented commutative
$k$-algebras
$$\HH^{[n]}(k[x])\cong k[x]\otimes
\underbrace  {
\B(k, \B(k,\cdots \B(k}
_{n \rm\ times}
,k[x],k)\cdots ,k) ,k)$$
where we take the diagonal of the multisimplicial set on the
right. This induces an isomorphism of the associated chain complexes.

Moreover, there is a map of augmented differential graded $k$-algebras
which is a quasi-iso\-morphism on the associated chain complexes
$$\HH^{[n]}(k[x])\cong k[x]\otimes
\underbrace  {
 \Tor^{\Tor^{\cdots^{\Tor
 ^{k[x]}(k,k)}\cdots}(k,k)}(k,k)
 }
_{n \rm\ times}
 \cong k[x]\otimes B'_{n+1},$$
for $B'_{n+1}$ from Definition \ref{defn:Bnprime}.
\end{thm}
Here the $\Tor$-expressions
and $B'_{n+1}$ are viewed as differential graded $k$-algebras with
respect to the trivial differential; thus it follows automatically
that the higher Hochschild homology groups of $k[x]$ are,
respectively, isomorphic to the part of them which has the appropriate
degree.
\begin{proof}
The first part of the claim is proved inductively.  From B\"okstedt's
decomposition we get
$$\HH^{[1]}(k[x])\cong
 \B^{k[x]}(k[x], k[x]\otimes C',k[x])\cong  k[x] \otimes \B (k, C',
k) \cong  k[x] \otimes \B(k, k[x],k)$$
 as 
simplicial augmented commutative  $k$-algebras. From this
decomposition and the same kind of splitting, we then get by Corollary
\ref{cor:iterated} that
$\HH^{[2]}(k[x])\cong k[x] \otimes \B(k,\B(k,k[x],k), k)$, and the
general statement follows by an iteration of this argument.

The second part uses the quasi-isomorphisms of differential graded
algebras from Section \ref{sec:lemmas}.  The point is that we have a
multiplicative quasi-isomorphism $\B(k, k[x], k)\simeq
\Lambda(\epsilon x) $, which means that we have  multiplicative
quasi-isomorphisms
$\B(k,\B(k,k[x],k), k) \simeq \B(k, \Lambda(\epsilon x),k) \simeq
\Gamma(\rho^0\epsilon x)$, and so on.  Thus instead of having a Veen-type
spectral sequence, which one can easily get for Hochschild homology
following the method that Veen used for topological Hochschild
homology, we have a complex of algebras.
\end{proof}

\begin{rem}
As mentioned before, we believe that an argument along the lines of
the above proof can show that Veen's spectral sequence collapses at
$E^2$ for certain commutative ring spectra. To this end one has to
establish that the higher topological Hochschild homology bar
constructions of these ring spectra are weakly equivalent via
multiplicative maps to the homotopy rings of the bar construction (taken
over the Eilenberg Mac Lane spectrum of $\fp$ rather than  over
$\fp$). Such an argument would be analogous to our proof that there
are multiplicative quasi-isomorphisms between
the bar constructions $\B(k,A,k)$ (for certain algebras $A$) and their
homology algebras as in Section \ref{sec:lemmas}.
\end{rem}

In low dimensions we can identify $\HH^{[n]}(\F_p[x])$ as
follows:  We know that Hochschild homology of $\F_p[x]$,
$\HH_*(\fp[x])$ is isomorphic to
$\Lambda_{\fp[x]}(\epsilon x)$ with $|\epsilon x|=1$. For Hochschild
homology of order two we obtain
$$\HH^{[2]}_*(\fp[x])\cong \G_{\F_p[x]}(\rho^0\epsilon x), \,
|\rho^0\epsilon x|= 2.$$
In the next step we get
$$\HH^{[3]}_*(\fp[x])\cong \Tor_{*,*}^{\G_{\F_p[x]}(\rho^0\epsilon
  x)}(\fp[x], \fp[x]) \cong
\left(\bigotimes_{k\geq 0} \Lambda_{\F_p[x]}(\epsilon \rho^k \epsilon
x)\right) \otimes \left(\bigotimes_{k\geq
 0} \G_{\F_p[x]}(\varphi^0 \rho^k \epsilon x)\right).$$
Using the flowcharts in Figure \ref{figure1} and Figure \ref{figure2} one can
explicitly calculate Hochschild homology of higher order.

\bigskip
 Specifying $k=\fp$ and using B\"okstedt's method again, if we
 consider the ring $\fpxl$ we obtain
$$\fpxl \otimes \fpxl \cong \fpxl \otimes C'',$$
as augmented commutative algebras, where $\fpxl$ is embedded as $\fpxl
\otimes k \subset\fpxl \otimes \fpxl$, and $C''\subset \fpxl\otimes
\fpxl$ is the $\F_p$-sub-algebra generated by the element
$x'=x\otimes 1 -1\otimes x$, with the relation $(x')^\pl =0$ so that
again  $C''= \fp  [x']/(x')^\pl  \cong \fpxl$.

\smallskip
We use this to get  a
calculation of the higher Hochschild homology groups of $\fpxl$.  In
\cite{P}, Pirashvili  calculated the $n$th higher Hochschild homology
groups of $k[x]/x^a$ for any $a$ when $n$  is odd and $k$ is a field
of characteristic zero using Hodge decomposition techniques. 

\begin{thm}\label{thm:kofxmodpl}
There is an isomorphism of  simplicial augmented commutative 
$\fp$-algebras
$$\HH^{[n]}(\fpxl)\cong \fpxl\otimes
\underbrace{\B(\fp, \B(\fp,\cdots \B(\fp}_{n \rm\ times} ,\fpxl,\fp
)\cdots ,\fp) ,\fp)$$
where we take the diagonal of the multisimplicial set on the
right. This induces an isomorphism of the associated chain complexes.

Moreover, there is a map of augmented differential graded
$\fp$-algebras which is a quasi-isomorphism on the associated chain
complexes
$$\HH^{[n]}(\fpxl)\cong \fpxl\otimes
\underbrace  {
 \Tor^{\Tor^{\cdots^{\Tor
 ^{\fpxl}(\fp,\fp)}\cdots}(\fp,\fp)}(\fp,\fp)
}
_{n \rm\ times}
 \cong \fpxl\otimes B''_{n+1},$$
for $B''_{n+1}$ from Definition \ref{defn:Bnprime}.  The $\Tor$-expressions
and $B''_{n+1}$ are again viewed as differential graded $\fp$-algebras
with a trivial differential.
\end{thm}

\section{\'Etale and Galois descent}
Ordinary Hochschild homology satisfies \'etale and Galois descent:
Weibel and Geller \cite{WG} showed that for an \'etale extension $A \ra B$ of
commutative $k$-algebras one has
$$ \HH_*(B) \cong \HH_*(A) \otimes_{A} B$$
and if $A \ra B$ is a Galois extension of commutative $k$-algebras in
the sense of Auslander-Goldman \cite{AG}
with finite Galois group $G$, then
$$ \HH_*(A) \cong \HH_*(B)^G.$$

We will show that these properties translate to higher order
Hochschild homology. In the following let $k$ be again an arbitrary
commutative unital ring and let $n$ be greater or equal to one.

\begin{thm}\label{thm:etale}
  \begin{enumerate}
\item[]
  \item
If $A$ is a commutative \'etale $k$-algebra, then $\HH_*^{[n]}(A)
\cong A$.
\item
If $A \ra B$ is an \'etale extension of commutative $k$-algebras, then
$$ \HH^{[n]}_*(B) \cong \HH^{[n]}_*(A) \otimes_{A} B.$$
\item
If $A \ra B$ is a $G$-Galois extension with $G$ a finite group, then
$$ \HH_*^{[n]}(A) \cong \HH^{[n]}_*(B)^G.$$
  \end{enumerate}
\end{thm}
\begin{proof}
The first claim follows from the second, but we also give a direct
proof: \'Etale $k$-algebras have Hochschild homology concentrated in degree
zero. Therefore Veen's spectral sequence yields
$$ \mathrm{Tor}_{p,q}^{\HH_*(A)}(A,A) \cong
\mathrm{Tor}_{p,q}^{A}(A,A) = A$$
in the $p=q=0$-spot and thus we get $\HH^{[2]}_*(A) =A$ concentrated
in degree zero. An iteration of this argument shows the claim for
arbitrary $n$.

For \'etale descent we deduce from Corollary \ref{cor:iterated} that
\begin{equation*}
\HH^{[2]}_*(B)\cong \mathrm{Tor}_{*}^{\HH^{[1]}_*(B)}(B,B) \cong
\mathrm{Tor}_{*}^{\HH_*^{[1]}(A) \otimes_A B}(A\otimes_A B,A\otimes_A B)
\cong \mathrm{Tor}_{*}^{\HH^{[1]} _*(A)}(A,A) \otimes_A B
\end{equation*}
and the latter is exactly $\HH^{[2]}_*(A)\otimes_A B$.   Note that
the maps $B = \HH_0(B) \ra
\HH_*(B)$ and $\HH_*(A) \ra \HH_*(B)$ used for the Weibel-Geller isomorphism
induce a map of graded
commutative rings $\HH_*(A) \otimes_A B \ra \HH_*(B)$, and the
argument above shows that our 
formulas for higher Hochschild homology are ring maps as well.

Iterating this argument, we get that
$\HH_*^{[n]}(B) \cong \HH_*^{[n]}(A) \otimes_A B$  for all $n$ as
graded commutative rings. 

Any $G$-Galois extension as above is in particular an \'etale
extension, so we get
$$ \HH^{[n]}_*(B) \cong \HH^{[n]}_*(A) \otimes_{A} B.$$
The $G$-action on the left hand side corresponds to the $G$-action on
the  $B$-factor on the right hand side and thus taking $G$-fixed points yields
$$  \HH^{[n]}_*(B)^G \cong \HH^{[n]}_*(A) \otimes_{A} (B^G) \cong
\HH^{[n]}_*(A) \otimes_{A} A \cong \HH^{[n]}_*(A).$$
\end{proof}

\section{Group algebras of finitely generated abelian groups}
The results of the preceding sections allow us to compute $\THH^{[n]}_*$ of
group algebras of finitely generated abelian groups over $\F_p$. If
$G$ is a finitely generated abelian group, then we
know from Section \ref{sec:splitting} that we need to determine
$\HH^{[n]}_*(\F_p[G])$ because $\THH^{[n]}_*(\F_p[G])$ is isomorphic
to the tensor product of $\THH^{[n]}_*(\F_p)$ and $\HH^{[n]}_*(\F_p[G])$. In
addition we know that $\F_p[G]$ can be written as a tensor product
$$ \F_p[G] \cong \F_p[\Z]^{\otimes r} \otimes \F_p[C_{q_1^{\ell_1}}]
\otimes \ldots \otimes \F_p[C_{q_s^{\ell_s}}]$$
where $r$ is the rank of $G$ and the $C_{q_i^{\ell_i}}$'s are the
torsion factors of $G$ for some primes $q_i$. As $\HH^{[n]}_*$ sends
tensor products to tensor products, we only have to determine the
tensor factors $\HH^{[n]}_*(\F_p[\Z])$ and $\HH^{[n]}_*(\F_p[C_{q_i^{\ell_i}}])$.

\begin{prop} \label{prop:factorsgroupalgs}
  \begin{itemize}
\item[]
  \item
For the group algebra $\F_p[\Z]\cong \F_p[x^{\pm 1}]$ we get
$$ \HH_*^{[n]}(\F_p[\Z])  \cong \F_p[x^{\pm 1}] \otimes B'_{n+1} .$$
\item
If $q$ is a prime not equal to $p$, then
$\HH^{[n]}_*(\F_p[C_{q^{\ell}}]) \cong \F_p[C_{q^{\ell}}]$ where the
latter is concentrated in homological degree zero.
\item
For $q=p$,
$$ \HH^{[n]}_*(\F_p[C_{p^\ell}]) \cong \F_p[x]/x^{p^\ell} \otimes B''_{n+1}.$$
  \end{itemize}
\end{prop}
\begin{proof}
The group algebra $\F_p[\Z]\cong \F_p[x^{\pm 1}]$ is \'etale over
$\F_p[x]$ and therefore by  Theorem
\ref{thm:etale} we obtain
$$\HH_*^{[n]}(\F_p[\Z])  \cong \HH_*^{[n]}(\F_p[x]) \otimes_{\F_p[x]}
\F_p[x^{\pm 1}]$$
and hence the first statement follows from Theorem \ref{thm:kofx}.

The group algebra $\F_p[C_{q^{\ell}}]$ is an \'etale algebra
over $\F_p$ for $q$ not equal to $p$, so  Theorem
\ref{thm:etale} also implies the second claim.

We know that $\F_p[C_{p^\ell}] \cong \F_p[x]/x^{p^\ell}$ because
$\F_p[x]/x^{p^\ell} - 1 =  \F_p[x]/(x-1)^{p^\ell}$.  Thus
$\HH^{[n]}_*(\F_p[C_{p^\ell}])$ is determined by Theorem
\ref{thm:kofxmodpl}.
\end{proof}
Thus if we express $G$ as
$$ G= \Z^r\times C_{p^{i_1}} \times \ldots \times C_{p^{i_a}} \times
C_{q_1^{j_1}} \times \ldots \times C_{q_b^{j_b}}$$
with $r,a,b\geq 0$, $i_s,j_t \geq 1$ and primes $q_i \neq p$, then we obtain
\begin{align*}
\THH & ^{[n]}_*(\F_p[G]) \cong
 \THH^{[n]}_*(\F_p) \otimes
\HH_*^{[n]}(\F_p[\Z]^{\otimes r} \otimes \bigotimes_{s=1}^a
\F_p[x]/x^{p^{i_s}} \otimes \bigotimes_{t=1}^b \F_p[C_{q_b^{j_b}}]) \\
\cong
& \THH^{[n]}_*(\F_p) \otimes
\bigl(\HH_*^{[n]}(\F_p[x])\otimes_{\F_p[x]}\F_p[x^{\pm 1}] \bigr) ^{\otimes r}
\otimes
\bigotimes_{s=1}^a \HH^{[n]}_*(\F_p[x]/x^{p^{i_s}})\otimes
\bigotimes_{t=1}^b \F_p[C_{q_b^{j_b}}].
\end{align*}

For instance, unravelling the definitions gives
\begin{align*}
\THH^{[2]}_*(\F_3[\Z \times \Z/6\Z]) \cong & \THH^{[2]}_*(\F_3) \otimes
\HH^{[2]}_*(\F_3[x]) \otimes_{\F_3[x]} \F_3[x^{\pm 1}] \otimes
\F_3[C_2] \otimes \HH^{[2]}_*(\F_3[x]/x^3) \\
\cong & \Lambda_{\F_3}(\epsilon y) \otimes (\F_3[x] \otimes
B'_3)\otimes_{\F_3[x]} \F_3[x^{\pm 1}] \otimes \F_3[C_2] \otimes
\F_3[x]/x^3 \otimes B''_3\\
\cong & \Lambda_{\F_3}(\epsilon y) \otimes \F_3[x^{\pm
  1}] \otimes  B'_3 \otimes   \F_3[C_2] \otimes
\F_3[x]/x^3 \otimes B''_3
\end{align*}
with $B'_3$ and $B''_3$ as explained in Definitions \ref{defn:Bnprime}
and \ref{defn:Bnprimeprime} and 
where $\epsilon y$ is a generator of degree three.
\appendix

\section{Code} \label{app:code}

Below is the Haskell code for generating possible differentials.  The code finds
all admissible words of a given length $n$ that fit into a particular portion of
the $E^2$ page and then looks for words that have consecutive degrees.  As the
shortest differential must go from an indecomposable to a primitive, we do not
generate any powers or products of words, as none of these can support a
shortest nonzero differential.

\begin{verbatim}
import System.Environment
import Data.List
import qualified Data.Set as S

main = do
  (prime:n:limit:_) <- getArgs
  putStrLn $ concat $ map pairToString
                          (possibleD (read n :: Integer)
                                     (read limit :: Integer)
                                     (read prime :: Integer))
-----------------

data VeenWord = M | E VeenWord | Rk VeenWord | Pk VeenWord
type Ppoly = [(Integer, (Integer,Integer))]

-- takes a sum and a list length and makes all lists of the length that
-- add up to at most m; this is the maximum degree of any particular
-- generator
varValueLists 0 m = [[]]
varValueLists 1 m = map (\a -> [a]) [0..m]
varValueLists n m = foldr (\l ls ->
                            let s = sum l
                            in (map (\a -> a:l) [0..m-s]) ++ ls)
                          [] (varValueLists (n-1) m)

makeKey M _ = "u"
makeKey (E w) l = "e" ++ (makeKey w l)
makeKey (Rk w) (a:as) = "r^" ++ (show a) ++ (makeKey w as)
makeKey (Pk w) (a:as) = "l^" ++ (show a) ++ (makeKey w as)
makeKey _ _ = error "Incorrect number of variables"


constantPoly n = [(n,(0,0))]
numVars = foldr (\(a,(_,c)) m -> if a == 0 || c == 0 then m
                                   else if c >= m then c else m) 0
compress p =
  let addup x [] = [x]
      addup x@(a,pair) ys@((a',pair'):l) =
        if pair == pair' then (a+a',pair):l else x:ys
  in foldr addup [] p

-- plugs in for variable number 1, shifts other variables down;
-- keep in mind that variable 3 is really the sum of three variables,v1,v2,v3
plugInV1 p v = compress $ map (\(a,(b,c)) -> if c >= 1
                                             then (a,(b+v,c-1))
                                             else (a,(b,c)))
                              p

plugInP :: Integer -> Ppoly -> Integer
plugInP prime p =
  let a ^^ n
        | n < 0 = error "Exponent must be positive"
        | n == 0 = 1
        | otherwise = a * (a ^^ (n-1))
  in if any (\(_,(_,c)) -> c /= 0) p
     then error "To plug in p you need to have no variables"
     else sum $ map (\(a,(b,_)) -> a * (prime ^^ b)) p

plugInAllVars :: Integer -> Ppoly -> [Integer] -> Integer
plugInAllVars prime p l = plugInP prime (foldl plugInV1 p l)

polyToString :: Ppoly -> String
polyToString =
  let monoToString (a,(b,c)) =
        (show a) ++ (if (b,c) == (0,0) then ""
                     else " P^{" ++ (if b /= 0 then (show b) ++ "+" else "")
                          ++ (if c /= 0
                              then "v_" ++ (show c)
                              else "") ++ "}")
  in (intercalate " + ") . (map monoToString)

addN n ((m,(0,0)):l) = (m+n,(0,0)):l
addN n l = (n,(0,0)):l
shiftBy1 = map (\(a,(b,c)) -> (a,(b+1,c)))
shiftByVar = map (\(a,(b,c)) -> (a,(b,c+1)))

degree :: VeenWord -> Ppoly
degree M = constantPoly 2
degree (E x) = addN 1 (degree x)
degree (Rk x) = shiftByVar $ addN 1 $ degree x
degree (Pk x) = shiftByVar $ addN 2 $ shiftBy1 $ degree x

bidegree :: VeenWord -> (Ppoly, Ppoly)
bidegree M = (constantPoly 0, constantPoly 2)
bidegree (E x) = (constantPoly 1, degree x)
bidegree (Rk x) = (shiftByVar $ constantPoly 1, shiftByVar $ degree x)
bidegree (Pk x) = (shiftByVar $ constantPoly 2,
                   shiftByVar $ shiftBy1 $ degree x)


makeAdmissibleWords n
  | n < 1 = error "makeAdmissibleWords needs positive integer"
  | n == 1 = [M]
  | otherwise =
      let words :: VeenWord -> [VeenWord] -> [VeenWord]
          words M l = (E M):l
          words w@(E _)  l = (Rk w):l
          words w@(Rk _) l = (E w):(Pk w):l
          words w@(Pk _) l = (E w):(Pk w):l
      in foldr words [] (makeAdmissibleWords (n-1))

--this takes a word and a pair of limits (which must be positive integers)
--and a prime p
--and generates all versions of the word and all powers of each version that
--will fit inside those limits
makeVersions :: VeenWord -> Integer -> Integer -> [(String,(Integer,Integer))]
makeVersions w maxdeg prime =
  let maxpow = (log (fromIntegral maxdeg))/(log (fromIntegral prime))
      estimate_bounds = floor(maxpow) :: Integer

      -- note that hom has at most one variable, which must have the same
      -- value as the first variable in inter
      (hom, inter) = bidegree w
      possibleVarValues = varValueLists (numVars inter) estimate_bounds
  in map (\l -> (makeKey w l, plugInAllVars prime hom   l,
                              plugInAllVars prime inter l))
         possibleVarValues

generateAllElts n maxdeg prime = concat $
                                  map (\w -> makeVersions w maxdeg prime)
                                      (makeAdmissibleWords n)

consecutivePairs l =
  [ (a,b,x-x') | a@(_,(x,y)) <- l, b@(_,(x',y')) <- l, x+y == x'+y'+1, x-x'>1]

possibleD n x prime = consecutivePairs $ generateAllElts n x prime

pairToString (a,b,deg) =
  let showThis (k,(x,y)) = k ++ (show (x,y))
  in (showThis a) ++ " ---> " ++ (showThis b) ++ ": " ++ (show deg) ++ "\n"
\end{verbatim}

\end{document}